\documentclass[final]{siamart190516}

\usepackage{amsfonts}
\usepackage{graphicx}
\usepackage{epstopdf}
\usepackage{algorithmic}
\ifpdf
  \DeclareGraphicsExtensions{.eps,.pdf,.png,.jpg}
\else
  \DeclareGraphicsExtensions{.eps}
\fi

\newsiamremark{remark}{Remark}
\newsiamremark{hypothesis}{Hypothesis}
\crefname{hypothesis}{Hypothesis}{Hypotheses}
\newsiamthm{claim}{Claim}
\newsiamthm{assumption}{Assumption}

\headers{Risk sensitive optimal stopping}{D. Jelito, M. Pitera and {\L}. Stettner}

\title{Risk sensitive optimal stopping}

\author{Damian Jelito\thanks{Institute of Mathematics, Jagiellonian University, Cracow, Poland,
  (\email{damian.jelito@im.uj.edu.pl}, \email{marcin.pitera@im.uj.edu.pl}); second author acknowledges research support by NCN grant 2016/23/B/ST1/00479.}
\and Marcin Pitera\footnotemark[1]
\and  {\L}ukasz Stettner\thanks{Institute of Mathematics, Polish Academy of Sciences, Warsaw, Poland,
  (\email{l.stettner@impan.pl});  research supported by NCN grant 2016/23/B/ST1/00479.}}

\usepackage{amsopn}
\DeclareMathOperator{\E}{\mathbb{E}}  

\usepackage{enumitem} 

\newcommand{\vep}{\varepsilon}

\def\cF{\mathcal{F}}
\def\cT{\mathcal{T}}
\def\bF{\mathbb{F}}
\def\bP{\mathbb{P}}
\def\bE{\mathbb{E}}
\def\bR{\mathbb{R}}
\def\bT{\mathbb{T}}
\def\bN{\mathbb{N}}

\makeatletter
\def\namedlabel#1#2{\begingroup
    #2%
    \def\@currentlabel{#2}%
    \phantomsection\label{#1}\endgroup
}
\makeatother

\begin{document}

\maketitle

\begin{abstract}
In this paper we consider discrete and continuous time risk sensitive optimal stopping problem. Using suitable properties of the underlying Feller-Markov process we prove continuity of the optimal stopping value function and provide formula for the optimal stopping policy. Next, we show how to link continuous time framework with its discrete time analogue. By considering suitable approximations, we obtain uniform convergence of the corresponding value functions.
\end{abstract}



\section{Introduction}\label{S:introduction}

In this paper we consider infinite time horizon risk sensitive optimal stopping problems of the form
\begin{equation}\label{eq:intro1}
\inf_{\tau} \E_x\left[\exp\left(\int_0^{\tau} g(X_s)ds+G(X_\tau)\right)\right],
\end{equation}
where $(X_t)$ is a Feller-Markov process taking values in a locally compact separable space $E$, $x\in E$ is a starting point, and $g,G:E\to\bR$ correspond to {\it running cost} and {\it terminal cost}, respectively. The functions $g$ and $G$ are assumed to be continuous, bounded, and non-negative, and $g$ is bounded away from zero by  a constant $c>0$.

For completeness, we consider separately discrete time and continuous time version of~\eqref{eq:intro1}. In both cases, we show that the optimal value function is continuous with respect to the starting point, and the optimal stopping time could be expressed as
\[
\tau=\inf\left\{t\in \bT: w(X_t)=e^{G(X_t)}\right\},
\]
where $w$ is the optimal stopping value and $\bT$ is the considered set of time points. We discuss relationships among continuous and discrete time problems and show how to uniformly approximate \eqref{eq:intro1} using discrete time value functions. It should be noted that the results presented in this paper are valid for both infinite and finite time horizons. In fact, most approximations are constructed by considering the limit of finite time problems.

The optimal stopping problem in the risk-neutral framework has been extensively studied in the literature; see e.g.~\cite{Shi1978,Zab1984,PesShi2006,PalSte2010} and references therein. Nevertheless, most of the developed risk-neutral methods do not transfer automatically to the risk-sensitive case. In fact, due to a complex nature of the problem, the coverage of risk-sensitive optimal stopping is very limited as it requires special treatment; see e.g.~\cite{Nag2007b, BauRie2017b,BauPop2018}. More generally, this applies to most risk-sensitive stochastic control problems in which non-linear objective criterion is used; see e.g.~\cite{Whi1990}.

Also, it should be noted that optimal stopping problems might be linked to the optimal impulse control framework. 
In particular, we refer to \cite{SadSte2002,HdiKar2011,PitSte2019} and \cite{BenLio1984,MenRob2017,PalSte2017} for risk-sensitive and risk-neutral impulse control framework discussion, respectively. In fact, in the companion paper \cite{JelPitSte2019b}, we show how to exploit this link and use results from this paper in order to prove the existence of a solution to the continuous time impulse control Bellman equation using its discrete time dyadic approximations. This shows how to apply probabilistic approach and dyadic framework from~\cite{PitSte2019} in the continuous time setting. Note that in \cite{Rob1981} a similar link is established for the classical risk-neutral case.

This paper is organised as follows: In Section \ref{S:preliminaries}, we provide a basic setup and introduce notation that will be used throughout the paper. Regularity properties of optimal stopping problems for discrete time case are established in Section~\ref{S:discrete_stopping}, and for the continuous time case in Section~\ref{S:continuous_stopping}. In particular, in Section~\ref{S:continuous_stopping} we show the main result of this paper, i.e. Theorem~\ref{th:w_continuity}. Finally, in Section~\ref{S:approximation_stopping} we link discrete time and continuous time frameworks by showing specific type of convergence.

\section{Preliminaries}\label{S:preliminaries}
Let $(\Omega,\cF,\bF,\bP)$ be a continuous time filtered probability space, where $\bF:=(\cF_t)_{t\in \bT}$, $\bT:=\bR_+$, $\cF_0$ is trivial, $\cF=\bigcup_{t\in\bT}\cF_t$, and the usual conditions are satisfied. We use $X:=(X_t)$ to denote a standard Markov process taking values in a locally compact separable metric space $E$ endowed with a metric $\rho$ and Borel $\sigma$-field $\mathcal{E}$; see Definition 4 in \cite[Section 1.4]{Shi1978}.

Moreover, we assume that $X$ satisfy the $C_0$-Feller property, i.e. given the transition semigroup $(P_t)$, for each $t\geq 0$, we get
\begin{equation}\label{as1}
P_t C_0(E) \subset C_0(E),
\end{equation}
where $C_0$ denotes continuous and bounded functions that are vanishing at infinity.

In a discretised version of the framework $(X_{n\delta })$ will be restricted to a discrete time grid with a predefined time step $\delta>0$ and the filtration $(\cF_{n\delta})_{n\in\bN}$. To ease the notation, if $\delta$ is clear from the context, we use $P$ to denote the transition operator $P_{\delta}$. It should be noted that most of the results presented for the discrete time framework require only the standard Feller property, i.e. $P C(E) \subset C(E)$, where $C(E)$ is the set of continuous bounded functions.

Let $\mathcal{T}$ denote the family of almost surely finite stopping times taking values in $\bT$, and let $\mathcal{T}_{m}\subset \cT$ denote the dyadic stopping times defined on the time grid $\{0,\delta_m, 2\delta_m, \ldots\}$, where $\delta_m:=(1/2)^m$ and $m\in\bN$. In particular, note that $\mathcal{T}_0$ denotes the family of finite stopping times for time step equal to $1$. For simplicity, we adapt the standard convention
\begin{equation}\label{eq:convention_stopping}
\E_x[Z_{\tau}]:=\liminf_{n\to \infty}\E_x[Z_{\tau\wedge n}],
\end{equation}
defined for any process $(Z_t)$ starting from $x\in E$ and $\tau\in\cT$. Also, we use convention $\sum_{i=0}^{-1}(\cdot)=0$, and sometimes write $\inf_{\tau}\E_x[Z_{\tau}]$ instead of $\inf_{\tau\in\cT}\E_x[Z_{\tau}]$ or $\inf_{\tau\in\cT_0}\E_x[Z_{\tau}]$.

Before we proceed, let us summarise some properties implied by~\eqref{as1} that will be used throughout the paper; see Proposition~\ref{pr:c0.feller}. 
\begin{proposition}\label{pr:c0.feller}
Let $X=(X_t)$ satisfy $C_0$-Feller property. Then,
\begin{enumerate}
\item $X$ satisfies $C$-Feller property.
\item For any $\vep>0$, compact set $\Gamma\subset E$, and $T>0$ there exists $R>0$ such that
\begin{equation}\label{Co}
\sup_{x\in \Gamma} \mathbb{P}_x\left[\sup_{t\in [0,T]} \rho(X_t,x)\geq R\right]\leq \vep.
\end{equation}
\item For any compact set $\Gamma \subset E$ and $r > 0$, we get
\begin{equation}\label{Co2}
\lim_{T\to 0} \sup_{x\in \Gamma} \mathbb{P}_x\left[\sup_{t\in [0,T]} \rho(X_t,x)\geq r\right] =0.
\end{equation}
\item For any $t\geq 0$ and $f,h\in C(E)$, the function
\begin{equation}\label{cont1}
P_t^f h(x):=\E_x\left[e^{\int_0^t f(X_s)ds}h(X_t)\right],\quad x\in E,
\end{equation}
is continuous and bounded.
\end{enumerate}
\end{proposition}

For more details and proofs of all properties given in Proposition~\ref{pr:c0.feller} we refer to \cite[Corollary 2.2]{PalSte2010}, \cite[Proposition 2.1]{PalSte2010}, \cite[Proposition 6.4]{BasSte2018}, and \cite[Lemma 4, Section II.5]{GikSko2004}.

\section{Discrete time risk sensitive optimal stopping}\label{S:discrete_stopping}
For simplicity, in this section we set $\delta=1$, and consider the linked discrete time process $(X_n)$. Considering stopping times from $\cT_0$ and following~\eqref{eq:intro1}, we define the discrete time optimal stopping value
\begin{equation}\label{eq1}
w(x):=\inf_{\tau} \E_x\left[e^{\sum_{i=0}^{\tau-1} g(X_i)+G(X_\tau)}\right],\quad x\in E,
\end{equation}
where functions $g$ and $G$ are assumed to be continuous, bounded, and non-negative, and $g$ is bounded away from zero by  a constant $c>0$. For any $n\in\bN$, the lower and upper approximates of $w$ are given by
\begin{align}
\underline w_n(x) &:=\inf_{\tau\leq n} \E_x\left[e^{\sum_{i=0}^{\tau-1} g(X_i)+1_{\{\tau<n\}}G(X_\tau) }\right],\label{eq.w.down}\\
\overline w_n(x) &:=\inf_{\tau\leq n} \E_x\left[e^{\sum_{i=0}^{\tau-1} g(X_i)+G(X_\tau)}\right].\label{eq.w.up}
\end{align}
It is easy to check that for any $n\in\mathbb{N}$ and $x\in E$ we get $\underline w_n(x) \leq w(x) \leq \overline w_n(x)$. Let us now show some properties of the functions $\underline w_n$ and $\overline w_n$ and link  them to Bellman operator
\begin{equation}\label{eq:Bellman_operator_discrete}
Sh(x) := e^{g(x)} Ph(x)\wedge e^{G(x)},\quad h\in C(E),
\end{equation}
corresponding to the Bellman equation
\begin{equation}\label{eq4b}
v(x)=e^{g(x)} Pv(x)\wedge e^{G(x)},
\end{equation}
where $v\in C(E)$; we will show later that the function given via~\eqref{eq4b} must be equal to~\eqref{eq1}, see Theorem~\ref{cor1}. 

\begin{proposition}\label{pr:wn}
Let $(\underline w_n)$ and $(\overline w_n)$ be given by~\eqref{eq.w.down} and~\eqref{eq.w.up}. Then,
\begin{enumerate}
\item The sequence $(\underline w_{n})$ is non-decreasing with initial value $\underline w_{0}(x)=1$, for $x\in E$. Moreover, for any $n\in\mathbb{N}$ and $x\in E$, we get $\underline w_{n}(x)=S\underline w_{n-1}(x)$ and the optimal stopping time for $\underline w_{n}$ is given by
\begin{equation}\label{eq3'}
\underline \tau_n:=\min\left\{i\geq 0: \underline w_{n-i}(X_i)=e^{G(X_i)}\right\}\wedge n.
\end{equation}

\item The sequence $(\overline w_{n})$ is non-increasing with initial value $\overline w_{0}(x)=e^{G(x)}$, for $x\in E$. Moreover, for any $n\in\mathbb{N}$ and $x\in E$, we get $\overline w_{n}(x)=S\overline w_{n-1}(x)$ and the optimal stopping time for $\overline w_{n}$ is given by
\begin{equation}\label{eq3''}
\overline \tau_n:=\min\left\{i\geq 0: \overline w_{n-i}(X_i)=e^{G(X_i)}\right\}.
\end{equation}

\end{enumerate}
\end{proposition}
\begin{proof}
For brevity, we show the proof only for $\underline{w}_n$ as the proof for $\overline{w}_n$ is analogous. The fact that 
$\underline w_{n+1}(x)=S \underline w_n(x)$, for $n\in \bN$ and $x\in E$, 
and the optimality of $\underline \tau_n$ may be shown using standard techniques, see e.g. \cite[Section 2.2]{Shi1978}. The monotonicity of $\underline w_{n}(x)$ follows easily from induction. Indeed, since $g(\cdot)\geq 0$ and $G(\cdot)\geq 0$, we get $\underline w_{1}(x)=e^{g(x)} P \underline w_{0}(x)\wedge e^{G(x)}\geq  1=\underline w_{0}(x).$ Hence, by induction assumption and monotonicity of $S$, we get $\underline w_{n+1}(x)=S \underline w_{n}(x)\geq S\underline w_{n-1}(x)=\underline w_{n}(x).$ 
\end{proof}

From Proposition~\ref{pr:wn} we see that the limits
\begin{equation}\label{eq.w.both} 
\underline w(x):=\lim_{n\to\infty}\underline w_{n}(x)\quad\textrm{and}\quad\overline w(x) :=\lim_{n\to\infty}\overline w_{n}(x)
\end{equation}
are well defined. Let us now outline some properties of $\underline w$ and $\overline w$.

\begin{proposition}\label{pr:w} Let $\underline w$ and $\overline w$ be given by~\eqref{eq.w.both}. Then,
\begin{enumerate}
\item The function $\underline w$ is lower semicontinuous and $\underline w(x)\leq w(x)$, for $x\in E$. Moreover, $\underline w$ is a solution to the Bellman equation \eqref{eq4b}.
\item The function $\overline w$ is upper semicontinuous and  $\overline w(x)\geq w(x)$, for $x\in E$. Moreover, $\overline w$ is a solution to the Bellman equation \eqref{eq4b}.
\end{enumerate}
\end{proposition}
\begin{proof}
By Proposition~\ref{pr:wn} we get that $\underline w_n$ could be seen as solutions to system of recursive equations
\begin{align}
\underline w_{n+1}(x)&=S \underline w_n(x),\quad \textrm{with } \underline w_{0}(x)=1,\label{eq:bakward}
\end{align}
given for $n\in\mathbb{N}$ and $x\in E$. Hence, for any $n\in\mathbb{N}$ the functions $\underline w_n$ are finite stopping values that are continuous by induction argument and the Feller property. Consequently, we get that the function $\underline w(x)$ is lower semicontinuous as a limit of non-decreasing continuous functions. Also, from \eqref{eq:bakward} it follows that $\underline w$ is a solution to the Bellman equation \eqref{eq4b}. Finally, inequality $\underline w(x) \leq w(x)$ follows from the fact that $\underline{w}_n(x)\leq w(x)$ for any $n\in \bN$ and $x\in E$. This concludes the proof for $\underline{w}$. The proof for $\overline{w}$ is similar and omitted for brevity.
\end{proof}
Now, we show the correspondence between~\eqref{eq4b} and~\eqref{eq1}.

\begin{proposition}\label{prop1} Let $v\in C(E)$ be a solution to the Bellman equation~\eqref{eq4b} such that $1\leq v(x)$ for any $x\in E$ and let $\tau_{v}:=\inf\left\{i\geq 0: v(X_i)=e^{G(X_i)}\right\}$. Then,
\begin{equation}\label{eq:prop1}
v(x)=\E_x\left[e^{\sum_{i=0}^{\tau_v-1} g(X_i)+G(X_{\tau_v})}\right],\quad x\in E,
\end{equation}
and $v$ coincides with value function $w$ defined in~\eqref{eq1}.
\end{proposition}
\begin{proof}
Let $v$ be a solution to~\eqref{eq4b} satisfying $1\leq v(x)$, for $x\in E.$ From the Bellman equation~\eqref{eq4b} we immediately get that the sequence of random variables $(z_n)$ given by
\begin{equation}\label{eq6```}
z_n:=e^{\sum_{i=0}^{\tau_v\wedge n-1} g(X_i)}v(X_{\tau_v\wedge n}),\quad n\in\bN,
\end{equation}
is a martingale under $\bP_x$, for any $x\in E$. Indeed, for any fixed $x\in E$, we get
\begin{align*} 
\E_x\left[z_{n+1}|\mathcal{F}_n\right] & =1_{\{\tau_v\leq n\}}\E_x\left[z_{n+1}|\mathcal{F}_n\right]+1_{\{\tau_v>n\}}\E_x\left[z_{n+1}|\mathcal{F}_n\right] \\
& = 1_{\{\tau_v\leq n\}} z_n + 1_{\{\tau_v>n\}}e^{\sum_{i=0}^{\tau_v\wedge n} g(X_i)}\bE_x\left[v(X_{n+1})|\mathcal{F}_n\right]  \\
& = 1_{\{\tau_v\leq n\}} z_n +1_{\{\tau_v>n\}}e^{\sum_{i=0}^{\tau_v\wedge n-1} g(X_i)}e^{g(X_n)}Pv(X_n)=z_n. 
\end{align*}
Consequently, we get
\begin{equation}\label{eqim}
v(x)=\E_x\left[z_0\right]=\E_x\left[z_n\right]=\E_x\left[e^{\sum_{i=0}^{\tau_v\wedge n-1} g(X_i)}v(X_{\tau_v\wedge n})\right].
\end{equation}
Using $1 \leq v(\cdot)\leq e^{G(\cdot)}$ and~\eqref{eqim}, we get
\begin{equation}
\E_x\left[e^{\sum_{i=0}^{\tau_v\wedge n-1} g(X_i)}\right] \leq \E_x\left[e^{\sum_{i=0}^{\tau_v\wedge n-1} g(X_i)}v(X_{\tau_v\wedge n})\right]=v(x)\leq e^{G(x)}.
\end{equation}
Therefore, recalling that $g(\cdot)\geq c>0$ and using Fatou Lemma, for any $x\in E$, we get
\begin{equation}\label{eq:prop1:integrability}
\E_x\left[e^{c \tau_v}\right]\leq \E_x\left[e^{\sum_{i=0}^{\tau_v-1} g(X_i)}\right]\leq e^{G(x)}<\infty.
\end{equation}
Now, observe that the sequence $\left(e^{\sum_{i=0}^{\tau_v\wedge n-1} g(X_i)}v(X_{\tau_v\wedge n})\right)_{n\in \mathbb{N}}$ is bounded by the random variable $e^{\sum_{i=0}^{\tau_v-1} g(X_i)} e^{\Vert G \Vert}$ that is integrable by~\eqref{eq:prop1:integrability}.  Thus, letting $n\to\infty$ in~\eqref{eqim} and recalling definition of $\tau_v$, we get
\begin{equation}\label{eqim'}
v(x)=\E_x\left[e^{\sum_{i=0}^{\tau_v-1} g(X_i)}v(X_{\tau_v})\right]=
\E_x\left[e^{\sum_{i=0}^{\tau_v-1} g(X_i)+G(X_{\tau_v})}\right],
\end{equation}
which concludes the proof of~\eqref{eq:prop1}; note that $v$ could be rewritten as
\begin{equation}\label{eq:opt_disc}
v(x)=\lim_{n\to \infty} \E_x\left[e^{\sum_{i=0}^{\tau_v\wedge n-1} g(X_i)+G(X_{\tau_v\wedge n})}\right].
\end{equation}

Using~\eqref{eqim'} we get that any solution $v$ to the Bellman equation~\eqref{eq4b} such that $v(\cdot)\geq 1$ satisfies $v(x)\geq w(x)$, $x\in E$, where $w$ is given in~\eqref{eq1}. In particular, note that by Proposition~\ref{pr:wn} we immediately get 
$\underline{w}(x)=w(x)$, $x\in E$.

Let us now show that $v(x) = w(x)$, for $x\in E$, where $w$ is given in~\eqref{eq1}. For any $x\in E$ we get $v(x)\leq e^{G(x)}=\overline w_0(x)$. Recalling the fact that $v$ is a solution to the Bellman equation~\eqref{eq4b}, i.e. $Sv=v$, using recursive property $S\overline w_{n}\equiv \overline w_{n+1}$, and monotonicity of $S$, for any $x\in E$ we get
\[
v(x) =\lim_{n\to\infty} S^nv(x) \leq \lim_{n\to\infty} S^n\overline w_{0}(x) = \lim_{n\to\infty} \overline w_{n}(x)=\overline w(x).
\]
Consequently, it is sufficient to show that $\overline{w}(x)\leq w(x)$, $x\in E$. Note that for any $n\in \mathbb{N}$ and $\hat\tau\in \cT_0$ we get
\[
\inf_{\tau \leq n}  \E_x\left[e^{\sum_{i=0}^{\tau-1} g(X_i)+G(X_\tau)}\right] \leq  \E_x\left[e^{\sum_{i=0}^{\hat\tau\wedge n-1} g(X_i)+G(X_{\hat\tau\wedge n})}\right].
\]
Thus, recalling convention~\eqref{eq:convention_stopping} and noting that
\begin{equation}\label{eq:ineq_overline_w}
\liminf_{n\to\infty}\inf_{\tau \leq n}  \E_x\left[e^{\sum_{i=0}^{\tau-1} g(X_i)+G(X_\tau)}\right] \leq  \inf_{\hat\tau}\liminf_{n\to\infty}\E_x\left[e^{\sum_{i=0}^{\hat\tau\wedge n-1} g(X_i)+G(X_{\hat\tau\wedge n})}\right],
\end{equation}
we get $\overline{w}\leq w$. This concludes the proof.
\end{proof}

\begin{corollary}\label{cor:unique_Bellman}
There is a unique solution to the Bellman equation~\eqref{eq4b} within the class of measurable functions v such that $1\leq v(x)$, $x\in E$.
\end{corollary}
We conclude this section by stating Theorem~\ref{cor1}, which shows equality of maps defined in~\eqref{eq.w.both} and~\eqref{eq1}, and outlines properties of $w$.
\begin{theorem}\label{cor1} The function $w$ defined in~\eqref{eq1} is continuous and bounded. Moreover, we get $\underline w\equiv w\equiv\overline w$, and the optimal stopping time is given by
\[
\hat \tau:=\inf\left\{i\geq 0: w(X_i)=e^{G(X_i)}\right\}.
\]
\end{theorem}
The proof of Theorem~\ref{cor1} is a direct result of Proposition~\ref{pr:w}, Proposition~\ref{prop1} and Corollary~\ref{cor:unique_Bellman}, and is omitted for brevity.

\section{Continuous time risk sensitive optimal stopping}\label{S:continuous_stopping}
Let us now translate results presented in Section~\ref{S:discrete_stopping} into continuous time setting. In analogy to~\eqref{eq1}, for stopping times from $\cT$, we define the function 
\begin{equation}\label{eq8}
w(x):=\inf_{\tau} \E_x\left[e^{\int_0^{\tau} g(X_s)ds+G(X_\tau)}\right].
\end{equation}
Also, for any $T\geq 0$, we define versions of~\eqref{eq.w.down} and~\eqref{eq.w.up} given by
\begin{align}
\underline{w}_T(x) & :=\inf_{\tau\leq T} \E_x\left[e^{\int_0^\tau g(X_s)ds + 1_{\{\tau<T\}}G(X_\tau) }\right],\label{eq:w^t_approx_from_below}\\
\overline{w}_T(x) & :=\inf_{\tau\leq T} \E_x\left[e^{\int_0^{\tau} g(X_s)ds+G(X_{\tau})}\right]\label{eq:w^T_approx_from_above}.
\end{align}
As before, for any $T\geq 0$ and $x\in E$, we get $\underline w_T(x) \leq w(x) \leq \overline w_T(x)$.  Our main goal is to prove continuity of the function $w$ defined in~\eqref{eq8} by exploiting properties of~\eqref{eq:w^t_approx_from_below} and~\eqref{eq:w^T_approx_from_above}. Before we do that, let us state some auxiliary results.

We start with a convergence result that will be used multiple times throughout this paper.

\begin{lemma}\label{lm:convergence_terminal}
Let $\tau$ be a bounded stopping time and let $(a_n)$ be a sequence of non-negative numbers such that $a_n\downarrow 0$ as $n\to\infty$. Then, for any sequence $(\tau_n)$ of stopping times satisfying $0\leq \tau_n-\tau \leq a_n$ and a compact set $\Gamma \subset E$, we get
\[
\lim_{n\to\infty} \sup_{x\in \Gamma} \E_x \left| e^{G(X_{\tau_n})}-e^{G(X_{\tau})}\right|=0.
\]
\end{lemma}

\begin{proof}
Let $\tau\in \cT$ be such that $\tau\leq T$, for some $T\in\bR_{+}$, and let $\vep>0$. Fix compact set $\Gamma \subset E$. Property~\eqref{Co} implies that there exists $R>0$ such that 
\begin{equation}\label{eq.C0.lem1}
\sup_{x\in \Gamma} \mathbb{P}_x\left[\sup_{t\in [0,T]} \rho(X_t,x)\geq R\right]\leq \vep.
\end{equation}
For brevity, we set $Z(t,s):=|e^{G(X_{t})}-e^{G(X_{s})}|$, $t,s\geq 0$. By~\eqref{eq.C0.lem1}, we get
\begin{align}\label{eq:lm:convergence_terminal:eq2}
\sup_{x\in \Gamma} \E_x \left[Z(\tau_n,\tau)\right] & \leq \sup_{x\in \Gamma} \E_x\left[1_{\{\rho(X_{\tau},x)\geq R\}}Z(\tau_n,\tau) +1_{\{\rho(X_{\tau},x)< R\}}Z(\tau_n,\tau) \right] \nonumber \\
 & \leq 2 \vep  e^{\Vert G\Vert} + \sup_{x\in \Gamma} \E_x\left[1_{\{\rho(X_{\tau},x)< R\}} Z(\tau_n,\tau)\right]  .
\end{align}
Set $B:=\{x\in E\colon \rho(\Gamma,x)\leq R+1\}$. Since $e^{G(\cdot)}$ is uniformly continuous on $B$, we can find $r>0$ such that
\begin{equation*}\label{eq:lm:convergence_terminal:continuity_tilde_g}
\sup_{z,y\in B\colon \rho(z,y)\leq r} |e^{G(y)}-e^{G(z)}|\leq \vep.
\end{equation*}
Recalling that $a_n\downarrow0$ and using~\eqref{Co2}, we know that there exists $n_0\in\bN$ such that, for any $n\geq n_0$, we get
\begin{equation*}\label{eq:lm:convergence_terminal:exit_convergence}
\sup_{x\in B} \mathbb{P}_x\left[\sup_{t\in [0,a_n]} \rho(x,X_t)\geq r\right]\leq \vep.
\end{equation*}
Also, using strong Markov property, we get
\[
\sup_{x\in \Gamma} \E_x\left[1_{\{\rho(X_{\tau},x)< R\}} Z(\tau_n,\tau) \right] \leq \sup_{x\in \Gamma} \E_x\left[1_{\{\rho(X_{\tau},x)< R\}} \E_{X_{\tau}} \left[\sup_{t\in[0,a_n]}Z(t,0) \right]\right].
\]
Consequently, noting that
\begin{align*}
\sup_{x\in \Gamma} \E_x\left[1_{\{\rho(X_{\tau},x)< R\}} \E_{X_{\tau}} \left[1_{\{\sup_{t\in[0,a_n]} \rho(x,X_t)< r\}} \sup_{t\in[0,a_n]}Z(t,0) \right]\right] & \leq \vep,\\
\sup_{x\in \Gamma} \E_x\left[1_{\{\rho(X_{\tau},x)< R\}} \E_{X_{\tau}} \left[1_{\{\sup_{t\in[0,a_n]} \rho(x,X_t)\geq r\}} \sup_{t\in[0,a_n]}Z(t,0) \right]\right] & \leq 2 \vep e^{\Vert G \Vert},
\end{align*}
we get
\begin{equation}\label{eq:lm:convergence_terminal:eq3}
\sup_{x\in \Gamma} \E_x\left[1_{\{\rho(X_{\tau},x)< R\}} Z(\tau_n,\tau) \right] \leq \vep+  2 \vep e^{\Vert G \Vert}.
\end{equation}
Combining~\eqref{eq:lm:convergence_terminal:eq2} with~\eqref{eq:lm:convergence_terminal:eq3}, for $n\geq n_0$, we get
\begin{equation}\label{eq:uniform.tau}
\sup_{x\in \Gamma} \E_x\left|e^{G(X_{\tau_n})}-e^{G(X_{\tau})}\right| \leq \vep (4  e^{\Vert G \Vert}+1),
\end{equation}
which concludes the proof.
\end{proof}
Next, the properties of maps $T\mapsto \underline w_{T}(x)$ and $T\mapsto \overline w_{T}(x)$, for a fixed $x\in\E$, are presented in Proposition~\ref{pr:w.time}.

\begin{proposition}\label{pr:w.time}
Let $\underline w_T$ and $\overline w_T$ be given by~\eqref{eq:w^t_approx_from_below} and~\eqref{eq:w^T_approx_from_above}. Then,
\begin{enumerate}
\item For any $x\in E$, the function $T\mapsto \underline w_{T}(x)$ is continuous and non-decreasing with initial value $\underline w_{0}(x)=1$.
\item For any $x\in E$, the function $T\mapsto \overline w_{T}(x)$ is continuous and non-increasing with initial value $\overline w_{0}(x)=e^{G(x)}$.
\end{enumerate}
\end{proposition}
\begin{proof}
First, we prove monotonic properties of $T\mapsto \underline w_{T}(x)$  and $T\mapsto \overline w_{T}(x)$.  Let $T,u\geq 0$ and let $\tau_{\vep}\leq T$ be an $\vep$-optimal stopping time for $\underline{w}_{T}(x)$. Then, using the fact that $g,G\geq 0$ we get
\begin{align*}
\underline{w}_{T-u}(x) & \leq \E_x\left[e^{\int_0^{\tau_{\vep}\wedge (T-u)} g(X_s)ds + 1_{\{\tau_{\vep} <T-u\}}G(X_{\tau_{\vep}})}\right]\\
&  \leq \E_x\left[e^{\int_0^{\tau_{\vep}} g(X_s)ds + 1_{\{\tau_{\vep} <T\}}G(X_{\tau_{\vep}})}\right] \leq \underline{w}_{T}(x) +\epsilon.
\end{align*}
Letting $\vep\to 0$, we conclude that $T\mapsto \underline w_{T}(x)$ is non-decreasing. The proof of monotonicity of $T\mapsto \overline{w}_T(x)$ is straightforward and omitted for brevity.

Second, we show the continuity property; we start with the proof of left-continuity of $T\mapsto \underline w_{T}(x)$. Let $\vep>0$. For any $u>0$, let $\tau_{\vep}^u\leq T-u$ be an $\vep$-optimal stopping time for $\underline{w}_{T-u}(x)$ and let 
\begin{alignat}{3}
\underline{z}_T(t) & :=e^{\int_0^{t\wedge T} g(X_s) ds + 1_{\{t<T\}} G(X_{t})},& \quad t &\geq 0.
\end{alignat}
Since $T\mapsto \underline w_{T}(x)$ is non-decreasing and $\underline{w}_T(x)\leq \E_x\left[ \underline{z}_T(\tau_{\vep}^u+u)\right]$, we get
\begin{align}\label{eq:down_uniform}
0\leq \underline{w}_T(x)-\underline{w}_{T-u}(x) & \leq \E_x\left| \underline{z}_T(\tau_{\vep}^u+u)-\underline{z}_{T-u}(\tau_{\vep}^u)\right|+\vep \nonumber\\
&\leq \sup_{\tau\leq T}\E_x\left| \underline{z}_T(\tau+u)-\underline{z}_{T-u}(\tau)\right|+\vep.
\end{align}
Now, let us show that, for any $T>0$ and $x \in E$, we get
\begin{equation}\label{eg:down.rc}
 \sup_{\tau\leq T} E_x\left| \underline{z}_T(\tau+u) -\underline{z}_{T-u}(\tau)\right|\to 0,\quad u\to 0.
\end{equation}
For any $x\in E$ and $\tau\leq T$ we get
\begin{align*}\label{eq:lm:boundedness_z_t_z_t-u:eq1}
\E_x\left| \underline{z}_T(\tau+u) -\underline{z}_{T-u}(\tau)\right| = & \E_x\left| e^{\int_0^{(\tau+u)\wedge T} g(X_s)ds} \left( e^{1_{\{\tau+u<T\}}G(X_{\tau+u})}-e^{1_{\{\tau<T-u\}}G(X_{\tau})}\right) \right. \nonumber \\
& \left. + e^{1_{\{\tau<T-u\}}G(X_\tau)+\int_0^{\tau \wedge (T-u)} g(X_s)ds}\left(e^{\int_{\tau \wedge (T-u)}^{(\tau+u)\wedge T} g(X_s)ds}-1\right)  \right| \nonumber \\
\leq & e^{T\Vert g \Vert}\E_x \left|e^{G(X_{\tau+u})}-e^{G(X_{\tau})}\right| +e^{\Vert G\Vert+T\Vert g\Vert}\left(e^{u\Vert g \Vert} -1\right).
\end{align*}
Recalling the proof of Lemma~\ref{lm:convergence_terminal} and noting that the upper bound~\eqref{eq:uniform.tau} depends only on the underlying sequence $(a_n)$, which in our case could be expressed through $u$, we get
\[
\sup_{\tau\leq T}\E_x \left|e^{G(X_{\tau+u})}-e^{G(X_{\tau})}\right|\to 0,\quad u\to 0.
\]
Consequently, since $e^{u\|g\|}-1\to 0$ as $u\to 0$, we conclude the proof of~\eqref{eg:down.rc}. As the choice of $\epsilon$ was arbitrary, we get left continuity of $T\mapsto \underline w_{T}(x)$. 

Next, let us show right-continuity of $T\mapsto \underline w_{T}(x)$. As in the first part of the proof, let $\tau_{\epsilon}\leq T$ be an $\vep$-optimal stopping time for $\underline{w}_T(x).$  Using monotonicity of $\underline w_{T}$, and boundedness of $g$ and $G$, we get
\begin{align}
\underline{w}_T(x) \leq \lim_{u\downarrow 0}\underline{w}_{T+u}(x) & \leq \lim_{u\downarrow 0} \E_x\left[e^{\int_0^{\tau_\vep+u} g(X_s)ds + 1_{\{\tau_\vep+u<T+u\}}G(X_{\tau_\vep+u})}\right] \nonumber \nonumber\\
& = \E_x\left[e^{\int_0^{\tau_\vep} g(X_s)ds + 1_{\{\tau_\vep<T\}}G(X_{\tau_\vep})}\right]\leq \underline{w}_T(x)+\vep;\label{eq:w_down_left_final}
\end{align}
note in the second line we used bounded convergence theorem and the fact that $(X_t)$ is right continuous. Letting $\vep\to 0$ we get right continuity of $T \to \underline{w}_T(x)$, for any $x\in E$.

The proof of continuity of $T\mapsto\overline{w}_T(x)$ is similar to the proof for $T\mapsto\underline{w}_T(x)$. For brevity, we only present an outline. Setting $\overline{z}_T(t)  :=e^{\int_0^{t\wedge T} g(X_s) ds + G(X_{t\wedge T})}, \, t\geq 0$ and using argument leading to~\eqref{eq:down_uniform}, for any $\varepsilon>0$, we get
\begin{align*}
0\leq \overline{w}_T(x)-\overline{w}_{T+u}(x) & \leq \sup_{\tau\leq T+u}\E_x\left| \overline{z}_T(\tau)- \overline{z}_{T+u}(\tau)\right|+\vep.
\end{align*}
By arguments similar to the ones used in the proof of~\eqref{eg:down.rc} we get right-continuity of $T\mapsto\underline{w}_T(x)$. Left-continuity could be obtained using the same reasoning as in \eqref{eq:w_down_left_final}, with $\underline{w}_{T+u}$ replaced by 
$\overline{w}_{T-u}$; note that quasi-left continuity of $(X_t)$ is required here.
\end{proof}
Next, we focus on the maps $x\mapsto \underline w_{T}(x)$ and $x\mapsto \overline w_{T}(x)$, defined for any fixed $T\geq 0$.
\begin{proposition}\label{prop:continuity_w_T}
Let $(\underline w_T)$ and $(\overline w_T)$ be given by~\eqref{eq:w^t_approx_from_below} and~\eqref{eq:w^T_approx_from_above}. Then,
\begin{enumerate}
\item For any $T\geq 0$, the function $x\mapsto \underline{w}_T(x)$ is continuous and bounded. Moreover, the optimal stopping time for $\underline w_T$ is given by
\begin{equation}\label{optstop}
\underline{\tau}_T:=\inf\left\{t\geq 0: \underline{w}_{T-t}(X_t)=e^{G(X_t)}\right\}\wedge T.
\end{equation}
\item For any $T\geq 0$, the function $x\mapsto \overline{w}_T(x)$ is continuous and bounded. Moreover, the optimal stopping time for $\overline w_T$ is given by
\begin{equation}\label{eq:prop:coninuity_w_T:optstop}
\overline{\tau}_T:=\inf\left\{t\geq 0:\overline{w}_{T-t}(X_t)=e^{G(X_t)}\right\}.
\end{equation}
\end{enumerate}
\end{proposition}

\begin{proof} 
The idea of the proof is to approximate the problem by its discrete time analogue. In the first part of the proof we focus on map $x\mapsto \underline{w}_T(x)$. For any $m\in\bN$ and $T\geq 0$, we set
\begin{equation}\label{eq:1step.1}
\underline{w}_T^m(x):=\inf_{ \tau \in {\cal T}_T^m} \E_x\left[ e^{\int_{0}^{\tau} g(X_s)ds+1_{\tau<T}G(X_\tau)}\right],\quad x\in E,
\end{equation}
where ${\cal T}^m_T$ is the family of stopping times taking values in $\left[0,{\tfrac{T}{2^m}},\tfrac{2T}{2^m}, \ldots,T\right]$. For transparency, we split the proof into three steps: (1) proof of continuity of $x\mapsto \underline{w}_T^m(x)$; (2) proof of continuity of $x\mapsto \underline{w}_T(x)$ using discrete time approximations; (3) proof of optimality of~\eqref{optstop}.

\medskip
\noindent {\it Step 1.} Continuity of $x\mapsto \underline{w}_T^m(x)$. Let us fix $m\in\bN$. Using operator $P^g_{\frac{T}{2^m}}$ defined in \eqref{cont1}, we consider the recursive sequence of functions
\begin{align*}
\widetilde{w}_T^0(x) & :=1,\\
\widetilde{w}_T^{j}(x) & := P_{\frac{T}{2^m}}^g \widetilde{w}_T^{j-1}(x)\wedge e^{G(x)}, \quad j=1, \ldots, 2^m.
\end{align*}
By property~\eqref{cont1} the function $\widetilde{w}_T^{j}$ is continuous, for $j=1,\ldots,2^m$. Using standard arguments (see e.g. \cite[Section 2.2]{Shi1978}) one can show that  $\underline{w}_T^m=\widetilde{w}_T^{2^m}$, which implies continuity of $x\mapsto \underline{w}_T^m(x)$.

\medskip
\noindent {\it Step 2.} Continuity of $x\mapsto \underline{w}_T(x)$. We show that $x\mapsto \underline{w}_T(x)$ could be approximated uniformly on compact sets by $x\mapsto \underline{w}_T^m(x)$, as $m\to\infty$. Let $\vep>0$ and $\tau_\vep\leq T$ be an $\vep$-optimal stopping time for $\underline{w}_T$. For any $m\in\bN$, we set
\[
\tau_{\vep}^m:=\inf\{\tau\in \cT_T^m \colon \tau\geq \tau_{\vep}\}= \textstyle \sum_{j=1}^{2^m}1_{\left\{\frac{T}{2^m}(j-1)< \tau_{\vep}\leq \frac{T}{2^m}j\right\}}\frac{T}{2^m}j.
\]
Noting that $\tau_{\vep}^m\leq T$, for any $x\in E$, we get
\begin{align}\label{cal1}
0&\leq  \underline{w}_T^m(x)-\underline{w}_T(x) \nonumber\\ 
&\leq  \E_x\left[ e^{\int_{0}^{\tau_{\vep}^m} g(X_s)ds+1_{\{\tau_{\vep}^m<T\}}G(X_{\tau_{\vep}^m})}\right] -\E_x\left[ e^{\int_{0}^{\tau_{\vep}} g(X_s)ds+1_{\{\tau_{\vep}<T\}}G(X_{\tau_{\vep}})}\right]+\vep \nonumber \\
&=  \E_x\left[ e^{\int_{0}^{\tau_{\vep}} g(X_s)ds}\left(e^{\int_{\tau_{\vep}}^{\tau_{\vep}^m} g(X_s)ds}-1\right)e^{1_{\{\tau_{\vep}^m<T\}}G(X_{\tau_{\vep}^m})}  \right] \nonumber \\
&\phantom{=} + \E_x\left[ e^{\int_{0}^{\tau_{\vep}} g(X_s)ds}\left(e^{1_{\{\tau_{\vep}^m<T\}}G(X_{\tau_{\vep}^m})}-e^{1_{\{\tau_{\vep}<T\}}G(X_{\tau_{\vep}})}\right)\right]+\vep  \nonumber \\
&\leq  e^{T\Vert g\Vert}\left(e^{\Vert G\Vert}\left(e^{\frac{T}{2^m}\Vert g\Vert}-1\right)+\E_x\left[e^{1_{\{\tau_{\vep}^m<T\}}G(X_{\tau_{\vep}^m})}-e^{1_{\{\tau_{\vep}<T\}}G(X_{\tau_{\vep}})}\right]\right)+\vep.
\end{align}
Noting that $\tau_{\vep}\leq \tau_{\vep}^m$ and recalling Lemma~\ref{lm:convergence_terminal} we get 
\begin{align*}
\E_x\left[e^{1_{\{\tau_{\vep}^m<T\}}G(X_{\tau_{\vep}^m})}-e^{1_{\{\tau_{\vep}<T\}}G(X_{\tau_{\vep}})}\right] &\leq \E_x\left[1_{\{\tau_{\vep}<T\}}\left(e^{G(X_{\tau_{\vep}^m})}-e^{G(X_{\tau_{\vep}})}\right)\right]\\
&\leq \E_x\left|e^{G(X_{\tau_{\vep}^m})}-e^{G(X_{\tau_{\vep}})}\right|\to 0, \quad m\to\infty.
\end{align*}
Also, the convergence is uniform (in $x$) on compact sets. Consequently, letting $\epsilon\to 0$, using \eqref{cal1}, and the fact that $\underline{w}_T^m$ is continuous, we get continuity of $\underline{w}_T$.

\medskip
\noindent {\it Step 3.} Optimality of~\eqref{optstop}. First, from continuity of $x\mapsto\underline{w}_T(x)$, for any $T\geq 0$, and Proposition~\eqref{pr:w.time} we get
\begin{equation}\label{eq:th:w^t:conv_t_x}
|\underline{w}_{T_n}(x_n)-\underline{w}_T(x)|\to 0, \quad n\to\infty,
\end{equation}
where $(T_n)\subset \mathbb{R}_+$ is a monotone and such that $T_n\to T$, and $(x_n)\subset E$ is such that $x_n\to x\in E$; this follows from Dini's theorem, as the convergence of $\underline{w}_{T_n}$ to $\underline{w}_T$ is uniform on compact sets. Second, for $t\in [0,T]$, we define
\begin{align*}
\underline{v}_T(t) &:=\underline{w}_{T-t}(X_t)e^{\int_0^t g(X_s)ds},\\
\underline{z}_T(t) &:=e^{\int_0^{t} g(X_s)ds + 1_{\{t<T\}}G(X_{t})}.
\end{align*}
Using techniques from \cite{Fak1971} one can show that $\underline{v}_T$ is the Snell envelope of $\underline{z}_T$. Hence, using \cite[Theorem 4]{Fak1970}, we get that
\begin{equation*}
\underline\tau^\vep_T:=\inf\left\{t\geq 0:  \underline{w}_{T-t}(X_t)e^{\int_0^t g(X_s)ds}\geq -\vep + \underline{z}_T(t)\right\}
\end{equation*}
is an $\vep$-optimal stopping time for $\underline{w}_T$. Thus, setting
\begin{equation}\label{eq:def:tau.nth}
{\hat{\underline\tau}}^\vep_T:=\inf\left\{t\geq 0:  \underline{w}_{T-t}(X_t)\geq (-\vep)\cdot e^{-\int_0^t g(X_s)ds} +e^{G(X_t)}\right\},
\end{equation}
we get $\underline\tau^\vep_T={\hat{\underline\tau}}^\vep_T\wedge T$. Now, noting that ${\hat{\underline\tau}}^{\vep_1}_T\geq {\hat{\underline\tau}}^{\vep_2}_T$, whenever $0\leq \vep_1\leq \vep_2$, we can define 
\[
{\hat{\underline\tau}}_T:=\lim_{\vep \downarrow 0} {\hat{\underline\tau}}^{\vep}_T\wedge T=\lim_{\vep \downarrow 0} \underline\tau^\vep_T.
\]
Using Fatou Lemma and quasi left continuity of $(X_t)$ we get
\begin{align}
\lim_{\vep \to 0} (\underline{w}_T(x)+\vep) & \geq \liminf_{\vep \to 0} \E_x \left[ e^{\int_0^{\underline\tau^\vep_T} g(X_s)ds + 1_{\{\underline\tau^\vep_T<T\}}G(X_{\underline\tau^\vep_T})} \right] \nonumber\\
& \geq \E_x\left[ e^{\int_0^{\hat{\underline\tau}_T} g(X_s)ds + 1_{\{\hat{\underline\tau}_T<T\}}G(X_{\hat{\underline\tau}_T})}\right] \nonumber \\
&\geq \underline{w}_T(x).\label{eq:part1.22}
\end{align}
Consequently, $\hat{\underline\tau}_T$ is the optimal stopping time for $\underline{w}_T(x)$. Finally, let us show that $\underline{\tau}_T=\hat{\underline\tau}_T$, where $\underline{\tau}_T$ is given by~\eqref{optstop}. Let $\epsilon>0$. On the set $\{{\hat{\underline\tau}}^\vep_T<T\},$ recalling definition \eqref{eq:def:tau.nth}, property \eqref{eq:th:w^t:conv_t_x}, continuity of $G$, and right-continuity of $(X_t)$, we get
\begin{equation}\label{eq:calc2}
\underline{w}_{T-{\hat{\underline\tau}}^\vep_T}\left(X_{{\hat{\underline\tau}}^\vep_T}\right)\geq (-\vep) \cdot e^{-\int_0^{{\hat{\underline\tau}}^\vep_T}g(X_s)ds}+e^{G(X_{{\hat{\underline\tau}}^\vep_T})}.
\end{equation}
Thus, on the set $\{\underline{\hat\tau}_T<T\}$, using~\eqref{eq:th:w^t:conv_t_x} and letting $\vep\downarrow 0$ in~\eqref{eq:calc2}, we get
\[
\underline{w}_{T-\hat\tau_T}(X_{\hat\tau_T})\geq e^{G(X_{\hat\tau_T})}.
\]
Since $\underline{w}_T(x)\leq e^{G(x)}$, for any $x\in E$ and $T>0$, on the set $\{\underline{\hat\tau}_T<T\}$ we also get
\begin{equation*}\label{eq:th:w^t:hattau}
\underline{w}_{T-\underline{\hat\tau}_T}(X_{\underline{\hat\tau}_T})= e^{G(X_{\underline{\hat\tau}_T})}.
\end{equation*}
Recalling definition of $\underline{\tau}_T$, we get $\underline{\tau}_T\leq \underline{\hat\tau}_T$. Finally, noting that $\underline\tau^\vep_T\leq  \underline{\tau}_T$, for any $\vep>0$, and letting $\vep\to 0$, we get $\underline{\tau}_T=\underline{\hat\tau}_T$ for $\underline{\tau}_T$ being the optimal stopping time. This concludes the first part of the proof.

The second part of the proof, i.e. the argument for $x\mapsto \overline{w}_T(x)$, is similar to the proof for $x\mapsto \underline{w}_T(x)$. For brevity, we only present an outline. In analogy to \eqref{eq:1step.1} we define
\begin{equation*}
\overline{w}_T^m(x):=\inf_{\tau\in {\cal T}^m_T} \E_x\left[e^{\int_0^{\tau} g(X_s)ds+G(X_{\tau})}\right], \quad m\in \mathbb{N}.
\end{equation*}
Using techniques presented in the first part of the proof, one could show that the map $x\mapsto \underline{w}_T^m(x)$ is continuous (Step 1) and use it to show continuity of $x\mapsto \underline{w}_T(x)$ (Step~2). Indeed, it is sufficient to combine  Lemma~\ref{cont1} with
\begin{align*}
0 &\leq  \overline{w}_T^m(x) - \overline{w}_T(x)\leq  e^{T \|g\|}\left((e^{{\frac{T}{2^m}}\|g\|}-1)e^{\|G\|} +\E_x\left|e^{G(X_{\tau_{\vep}^m})}-e^{G(X_{\tau_{\vep}})}\right|\right)+\vep,
\end{align*}
where $\tau_\vep$ is an $\epsilon$-optimal stopping time and  $\tau_{\vep}^m$ is the (upper) $\cT^m_T$-dyadic approximation of $\tau_{\vep}^m$. To prove equivalent of Step 3, we consider $\overline{z}_T(t) :=e^{\int_0^{t} g(X_s)ds+G(X_{t})}$, and $\overline{v}_T(t):=\overline{w}_{T-t}(X_{t})e^{\int_0^{t} g(X_s)ds}$, defined for $t\in [0,T]$. Using arguments similar as in Step 3, we get that $\overline\tau_T^{\vep}  :=\inf\{t\geq 0: \overline{w}_{T-t}(X_{t})\geq -\vep e^{-\int_0^{t} g(X_s)ds} + e^{G(X_{t})}\}$ is an $\vep$-optimal stopping time for $\overline{w}_T(x)$, and we can define $\hat{\overline\tau}_T:=\lim_{\vep\to 0}\overline\tau_T^{\vep}$. Next, using similar reasoning as in \eqref{eq:part1.22}, we can show that $\hat{\overline\tau}_T$ is optimal for $\overline{w}_T(x)$. Finally, the proof of $\overline{\tau}_T=\hat{\overline\tau}_T$ is similar to the proof of $\underline{\tau}_T=\hat{\underline\tau}_T$ provided in Step 3.
\end{proof}
\begin{remark}\label{rm:joint_continuity_w_T(x)}
Observe that in Proposition~\ref{prop:continuity_w_T} joint continuity of $(T,x)\mapsto \underline{w}_T(x)$ and $(T,x)\mapsto \overline{w}_T(x)$ 
was shown; see~\eqref{eq:th:w^t:conv_t_x}.
\end{remark}

\begin{remark}\label{rm:g_G_general}
For $\overline{w}_T(x)$, the statements of Proposition~\ref{pr:w.time} and Proposition~\ref{prop:continuity_w_T} remain valid even if we consider generic (possibly negative) functions $\tilde g, \tilde G\in C(E)$ instead of $g$ and $G$. For $G$, it is enough to note that we can multiply both sides of \eqref{eq:w^T_approx_from_above} by $e^{\Vert\tilde G\Vert}$. For $g$, non-negativity was used only to show that for any $\tau \in \mathcal{T}$ we get $\E_x\left| e^{\int_\tau^{\tau+h} g(X_s) ds}-1\right|\to 0$ as $h\downarrow 0$. Noting that $|e^z-e^y|\leq e^{\max(z,y)}|z-y|$, $z,y\in \mathbb{R}$, for generic $\tilde g$ we get $\E_x\left| e^{\int_\tau^{\tau+h} \tilde g(X_s) ds}-1\right|\leq e^{h\Vert \tilde g\Vert} h\Vert \tilde g\Vert\to 0$, as $h\downarrow 0$.
\end{remark}

Next, in analogy to~\eqref{eq.w.both}, we define the limits
\begin{equation}\label{eq:w.both2}
\underline w(x):=\lim_{T\to\infty}\underline w_{T}(x)\quad\textrm{and}\quad\overline w(x) :=\lim_{T\to\infty}\overline w_{T}(x);
\end{equation}
note that due to Proposition~\ref{pr:w.time} those functions are well defined. Let us outline some properties of $\underline{w}$ and $\overline{w}$.

\begin{proposition}\label{pr:w.both} Let $\underline w$ and $\overline w$ be given by~\eqref{eq:w.both2}. Then,
\begin{enumerate}
\item The function $\underline w$ is lower semicontinuous and $\underline w(x)\leq w(x)$, for $x\in E$.
\item The function $\overline w$ is upper semicontinuous and  $\overline w(x)\geq w(x)$, for $x\in E$.
\end{enumerate}
\end{proposition}
\begin{proof}
Using Proposition~\ref{pr:w.time} and Proposition~\ref{prop:continuity_w_T}, we get that $\underline{w}$ is a non-decreasing limit of continuous functions $\underline{w}_T$, hence it is is lower semicontinuous.  Similarly, $\overline{w}$ is upper semicontinuous as a non-increasing limit of continuous functions $\overline{w}_T$.
\end{proof}
In Theorem~\ref{th:w_continuity} we show that the maps defined by~\eqref{eq:w.both2} are equal to $w$, which is used to prove continuity of $w$.

\begin{theorem}\label{th:w_continuity}
The function $w$ defined in~\eqref{eq8} is continuous and bounded. Moreover, we get $\underline w\equiv w\equiv\overline w$, and the optimal stopping time is given by
\begin{equation}\label{eq:th:w_continuity:optstop}
\hat{\tau}:=\inf\left\{t\geq 0: w(X_t)=e^{G(X_t)}\right\}.
\end{equation}
\end{theorem}

\begin{proof}
The proof of boundedness of $w$ is straightforward and omitted for brevity.  For transparency, we split the rest of the argument into three steps: (1) proof of upper semicontinuity of $w$; (2) proof of lower semicontinuity of $w$; (3) proof of optimality of~\eqref{eq:th:w_continuity:optstop}.

\medskip
\noindent {\it Step 1.} Upper semicontinuity of $w$. Following similar reasoning as in the proof of Proposition~\ref{prop1} we get $w(x)= \overline{w}(x)$, for $x\in E$; cf.~\eqref{eq:ineq_overline_w}. Thus, recalling Proposition~\ref{pr:w.both}, we get that $w$ is upper semicontinuous

\medskip
\noindent {\it Step 2.}  Lower semicontinuity of $w$. It is sufficient to show that $w(x)= \underline{w}(x)$, $x\in E$, and use Proposition~\ref{pr:w.both}. For any $T>0$, let $\underline{\tau}_T$ be the optimal stopping time for $\underline{w}_T$, given by the formula~\eqref{optstop}. Define 
\[
{\hat{\underline\tau}}_T:=\inf\left\{t\geq 0: \underline{w}_{T-t}(X_t)\geq e^{G(X_t)}\right\}
\]
and observe that $\underline{\tau}_T={\hat{\underline\tau}}_T\wedge T$. Recalling that $g(\cdot)\geq c>0$ and $G(\cdot)\geq 0$, we get
\begin{equation*}
e^{\|G\|}\geq \underline{w}_{T}(x)=\E_x\left[e^{\int_0^{\underline{\tau}_T} g(X_s)ds + 1_{\{\underline{\tau}_T<T\}}G(X_{\underline{\tau}_T})}\right]\geq \E_x\left[1_{\{\underline{\tau}_T=T\}}\right]e^{cT}.
\end{equation*}
Consequently, considering only integer time-points, we get
\begin{equation}\label{eq:cl1}
\sum_{n=1}^\infty \bP_x\left[\underline{\tau}_{n}=n\right]\leq \sum_{n=1}^\infty{\frac{e^{\|G\|}}{e^{cn}}}<\infty.
\end{equation}
Moreover, by Proposition~\ref{pr:w.time}, we get ${\hat{\underline\tau}}_{n+1}\leq {\hat{\underline\tau}}_n$. Thus,  for any $n\in\bN$, on the set $\{{{\underline\tau}}_n<n\}$, we get ${\hat{\underline\tau}}_{n}={{\underline\tau}}_{n}$, ${\hat{\underline\tau}}_{n+1}={{\underline\tau}}_{n+1}$, and consequently ${{\underline\tau}}_{n+1}\leq {{\underline\tau}}_n$. Thus, by Borel-Cantelli Lemma applied to \eqref{eq:cl1}, we get that the limit
\[
\underline{\tau}:=\lim_{n\to \infty} 1_{\{{{\underline\tau}}_n<n\}}\underline{\tau}_n
\]
is well defined. Using right continuity of $(X_t)$, Fatou Lemma, and noting that $\underline{w}\leq w$, we get
\begin{align}\label{eq:th:w_continuity:ineq2}
\E_x\left[e^{\int_0^{\underline{\tau}} g(X_s)ds+G(X_{\underline{\tau}})}\right] & = \E_x\left[\lim_{n\to \infty}1_{\{{{\underline\tau}}_n<n\}} \left(e^{\int_0^{\underline{\tau}_n} g(X_s)ds+1_{\{\underline{\tau}_n<n\}}G(X_{\underline{\tau}_n})}\right)\right] \nonumber \\
& \leq \liminf_{n\to\infty} \underline{w}_n(x) = \underline{w}(x)\leq w(x)<\infty.
\end{align}
Thus, recalling convention~\eqref{eq:convention_stopping}, using bounded convergence theorem with quasi-left continuity of $(X_t)$, and~\eqref{eq:th:w_continuity:ineq2}, we get
\[
w(x)\leq \liminf_{T\to\infty} \E_x\left[e^{\int_0^{\underline{\tau}\wedge T} g(X_s)ds+G(X_{\underline{\tau}\wedge T})}\right] = \E_x\left[e^{\int_0^{\underline{\tau}} g(X_s)ds+G(X_{\underline{\tau}})}\right] \leq \underline{w}(x)\leq w(x),
\]
which implies $\underline{w}(x)=w(x)$, for $x\in E$. Using Proposition~\ref{pr:w.time} and Proposition~\ref{prop:continuity_w_T} we get that $w$ is lower semicontinuous as an increasing limit of continuous functions $\underline{w}_T$. 

\medskip
\noindent {\it Step 3.} Optimality of~\eqref{eq:th:w_continuity:optstop}. For fixed $T\geq0$ and $t\geq 0$, we define
\begin{align*}
\overline{v}_T(t) & :=\overline{w}_{T-t\wedge T}(X_{t\wedge T})e^{\int_0^{t\wedge T}g(X_s)ds},\\
v(t) & :=w(X_{t})e^{\int_0^{t}g(X_s)ds}.
\end{align*}
Note that for any fixed $t\geq 0$ we get that $\overline{v}_T(t)$ converges to $v(t)$ as $T\to\infty$; this follows from Dini's theorem, property $\overline{w}_T \downarrow w$ as $T\to \infty$, and continuity of $\overline{w}_T$ and $w$. Moreover, using standard results one can show that the process $(\overline{v}_T(t\wedge \overline{\tau}_T))$ is a martingale, where $\overline{\tau}_T$ is given by~\eqref{eq:prop:coninuity_w_T:optstop}; see e.g. Theorem 1 in \cite[Section 5.3]{Shi2019}. Since by Proposition~\ref{pr:w.time} the map $T\mapsto \overline{\tau}_T$ is increasing, we can set
\[
\overline{\tau}:=\lim_{T\to\infty}\overline{\tau}_T.
\]
Let us now show that the process $(v(t\wedge\overline{\tau}))_{t\geq 0}$ is a martingale.

By martingale property of $(\overline{v}_T(t\wedge \overline{\tau}_T))$ for any $t,h\geq 0$, we get
\begin{equation}\label{eq:th:w_continuity:condexp}
\bE_{x}\left[ \overline v_T((t+h)\wedge \overline\tau_T)|\cF_t\right] = \overline v_T(t\wedge \overline\tau_T).
\end{equation}
From Dini's theorem and quasi-left continuity of $(X_t)$ we get that
\[
\overline v_T((t+h)\wedge \overline\tau_T)\to  v((t+h)\wedge \overline\tau),\quad T\to\infty.
\]
Thus, letting $T\to \infty$ in~\eqref{eq:th:w_continuity:condexp} and using bounded convergence theorem, we get
\begin{equation}\label{eq:th:w_continuity:martingale1}
\bE_{x}\left[  v((t+h)\wedge \overline\tau)|\cF_t\right] =v(t\wedge \overline\tau),
\end{equation}
which concludes the proof of the martingale property of $(v(t\wedge\overline{\tau}))$.

Next, let us show that $\overline{\tau}$ is optimal for $w$. Using \eqref{eq:th:w_continuity:martingale1}, for any $t\geq 0$, we get 
\begin{equation}\label{eq:th:w_continuity_martingale}
w(x)=\E_x \left[ w(X_{t\wedge\overline{\tau}})e^{\int_0^{t\wedge\overline{\tau}}g(X_s)ds}\right].
\end{equation}
By Fatou Lemma, recalling that $0<c\leq g(\cdot)$ and $1\leq w(\cdot)$, we get $\overline{\tau}<\infty$, since
\[
\E_x e^{\overline{\tau} c} \leq \liminf_{t\to \infty} \E_x e^{(t \wedge \overline{\tau} ) c}\leq \liminf_{t\to \infty} \E_x \left[ w(X_{t\wedge\overline{\tau}})e^{\int_0^{t\wedge\overline{\tau}}g(X_s)ds}\right] = w(x).
\]
Letting $t\to\infty$ in~\eqref{eq:th:w_continuity_martingale}, by Fatou lemma and quasi-left continuity of $(X_t)$, we get
\[
w(x)\geq \E_x \left[ w(X_{\overline{\tau}})e^{\int_0^{\overline{\tau}}g(X_s)ds}\right].
\]
Consequently, by bounded convergence theorem, we get
\begin{equation}\label{eq:th:w_continuity:dynprog}
w(x)=\E_x \left[ w(X_{\overline{\tau}})e^{\int_0^{\overline{\tau}}g(X_s)ds}\right].
\end{equation}
Thus, to conclude the proof that $\overline{\tau}$ is optimal for $w$ it is sufficient to show that we can replace $w(X_{\overline{\tau}})$ by $e^{G(X_{\overline{\tau}})}$ in \eqref{eq:th:w_continuity:dynprog}.

Using right continuity of $(X_t)$ and recalling~\eqref{eq:prop:coninuity_w_T:optstop}, we get
\begin{equation}
\overline{w}_{T-\overline{\tau}_T}(X_{\overline{\tau}_T})=e^{G(X_{\overline{\tau}_T})}.
\end{equation}
By Dini's theorem, letting $T\to \infty$ and using quasi-left continuity of $(X_t)$, we get
\begin{equation}\label{eq:cl2}
w(X_{\overline{\tau}})=e^{G(X_{\overline{\tau}})}.
\end{equation}
Using~\eqref{eq:th:w_continuity:dynprog}, we get $w(x)=\E_x \left[ e^{\int_0^{\overline{\tau}}g(X_s)ds+G(X_{\overline{\tau}})}\right]$, i.e. $\overline{\tau}$ is the optimal stopping time for $w$.

Finally, from \eqref{eq:cl2} we get $\hat{\tau}\leq \overline{\tau}$, where $\hat{\tau}$ is defined by~\eqref{eq:th:w_continuity:optstop}. Moreover, noting that $\overline{w}_T\geq w$, we get $\overline{\tau}_T \leq \hat\tau$. Letting $T\to\infty$, we also have $\overline{\tau} \leq \hat\tau$. This implies $\overline{\tau}=\hat\tau$ and concludes the proof.
\end{proof}

\begin{remark}\label{rm:convention}
As an auxiliary result from the proof of Theorem~\ref{th:w_continuity} we get that 
\begin{equation}\label{eq:infimum_attained}
w(x)= \E_x\left[ e^{\int_0^{\hat{\tau}}g(X_s)ds+G(X_{\hat{\tau}})}\right],
\end{equation}
where $\hat\tau=\inf\left\{t\geq 0: w(X_t)=e^{G(X_t)}\right\}$ and the right hand side of~\eqref{eq:infimum_attained} may be understood as the standard expectation, i.e. without the convention~\eqref{eq:convention_stopping}. 
\end{remark}

\begin{remark}\label{rm:martingale}
Using similar technique as in the proof of Theorem~\ref{th:w_continuity} one can show that  the process $v(t)=e^{\int_0^{t} g(X_s) ds}w(X_{t}), t\geq 0$, is a submartingale; note that from~\eqref{eq:th:w_continuity:martingale1} it follows that $(v(t\wedge \hat\tau))$ is a martingale.
\end{remark}

\section{Approximation of optimal stopping problems}\label{S:approximation_stopping}
The main goal of this section is to show that~\eqref{eq8} could be approximated by
\begin{equation}\label{eq:opt_stop_simp_discrete}
w_m(x):=\inf_{\tau\in \mathcal{T}_m}\E_x \left[ e^{\int_0^\tau g_m(X_s)ds +G_m(X_\tau)}\right],\quad m\in \bN,
\end{equation}
where $(g_{m})$ and $(G_m)$ are fixed sequences of functions from $C(E)$. We assume that $g_m\uparrow g$ as $m\to \infty$ with $g_0(\cdot)\geq c_0> 0$ and $G_m\to G$ uniformly with $G_m(\cdot)\geq 0$. We show that the sequence given in~\eqref{eq:opt_stop_simp_discrete} converges uniformly on compact sets to~\eqref{eq8}, as $m\to\infty$.

Note that this provides a link between continuous and discrete time optimal stopping framework. Also, it might help to establish existence of solution to the continuous time Bellman equation for impulse control problem by using its discrete time approximations; see~\cite{JelPitSte2019b} for details.

\begin{theorem}\label{th:approximation_discrete}
Let $w$ and $w_m$ be given by~\eqref{eq8} and~\eqref{eq:opt_stop_simp_discrete}, respectively. Then, $w_m\to w$, as $m\to \infty$, uniformly on compact sets.
\end{theorem}

\begin{proof}
Noting that 
\[
\vert w(x) -w_m(x)\vert \leq  \vert w(x) -v_m(x)\vert+ \vert v_m(x)-w_m(x)\vert,
\]
where, $x\in E$, $m\in\bN$, and the function $v_m\colon E\to \bR$ is given by
\[
v_m(x)  :=\inf_{\tau\in \mathcal{T}_m}\E_x\left[ e^{\int_0^\tau g_m(X_s)ds +G(X_\tau)}\right],
\]
it is sufficient to show that $\vert w(x) -v_m(x)\vert$ and $\vert v_m(x)-w_m(x)\vert$ converge to zero uniformly on compact sets. Moreover, noting that $\mathcal{T}_m\subset \mathcal{T}$ and $g_m\uparrow g$, we get 
\[
\underline v_m(x) \leq v_m(x) \leq \overline v_m(x),
\]
where the lower and upper bounds of $v_m$ are given by
\[
\underline v_m(x)  :=\inf_{\tau\in \mathcal{T}}\E_x \left[e^{\int_0^\tau g_m(X_s)ds+G(X_\tau)}\right],\quad\quad \overline v_m(x)  :=\inf_{\tau\in \mathcal{T}_m}\E_x \left[e^{\int_0^\tau g(X_s)ds+G(X_\tau)}\right].
\]
Hence, for the convergence of $v_m$ to $w$ it is sufficient to show that $\underline v_m$ and $\overline v_m$ both converge to $w$ uniformly on compact sets. For transparency, we split the rest of the proof into three steps: (1) proof of $|\underline v_m -w|\to 0$; (2) proof of $|\overline v_m-w|\to 0$; (3) proof of $|v_m-w_m|\to 0$.

\medskip
\noindent {\it Step 1.} We show that $\vert {\underline v_m}-w\vert\to 0$ uniformly on compact sets. For any $x\in E$, we get $w(x)\geq {\underline v_m}(x)$ and ${\underline v}_{m+1}(x)\geq {\underline v_m}(x)$. Thus, the limit ${\underline v}(x):=\lim_{m \to\infty}{\underline v_m}(x)$ is well defined, and 
\begin{equation}
w(x)\geq {\underline v}(x), \quad x\in E.
\end{equation}
Using Theorem~\ref{th:w_continuity}, we get that
\[
\underline\tau_m:=\inf\left\{t\geq 0: {\underline v_m}(X_t) \geq e^{G(X_t)}\right\} 
\]
is the optimal stopping time for ${\underline v_m}$. Since $\underline\tau_{m+1} \leq \underline\tau_{m}$, the limit $\underline\tau:=\lim_{m\to\infty}\underline \tau_m$ is well defined. Then, using Remark \ref{rm:convention}, Fatou Lemma, and right continuity of $(X_t)$, we get
\begin{align*}
w(x) & \geq \lim_{m\to \infty} {\underline v_m}(x) = \lim_{m\to \infty} \E_x  \left[e^{\int_0^{\underline\tau_m} g_m(X_s)ds+G(X_{\underline\tau_m})}\right]\\
& \geq \E_x \left[ \liminf_{m\to \infty} e^{\int_0^{\underline\tau_m} g_m(X_s)+G(X_{\underline\tau_m})}\right]=\E_x \left[ e^{\int_0^{\underline\tau} g(X_s)ds +G(X_{\underline\tau})}\right]\\
& = \liminf_{T\to\infty} \E_x \left[ e^{\int_0^{\underline\tau\wedge T} g(X_s)ds+G(X_{\underline\tau\wedge T})}\right]\geq w(x),
\end{align*}
where, in the last line, we used bounded convergence theorem. Therefore, we get $w(x)= {\underline v}(x)$, $x\in E$. 

Recalling Theorem~\ref{th:w_continuity} we get that for any $m\in \bN$ functions $\underline{v}_m$ and $w$ are continuous. Using the fact that convergence ${\underline v_m}(x) \to w(x)$ is monotone, by Dini's theorem we conclude that ${\overline v_m} \to w$ uniformly on compact sets.

\medskip
\noindent {\it Step 2.} We show that $\vert{\overline v_m}-w \vert\to 0$ uniformly on compact sets. Clearly, for any $x\in E$, we get $w(x)\leq {\overline v_m}(x)$ and ${\overline v}_{m+1}(x)\leq {\overline v_m}(x)$. Thus, the limit ${\overline v}(x):=\lim_{m \to\infty}{\overline v_m}(x)$ is well defined, and $w(x)\leq {\overline v}(x)$, $ x\in E$.

Now, we show that ${\overline v}(x)\leq w(x)$, $x\in E$. Let $\hat\tau$ denote the optimal stopping time for $w$ given in \eqref{eq:th:w_continuity:optstop}, and let $\hat\tau_m$ denote its ${\cal T}_m$ approximation given by
\[
\hat \tau_m:=\inf\{\tau\in \cT_m \colon \tau\geq \hat\tau\}= \textstyle \sum_{j=1}^{\infty}1_{\left\{\frac{j-1}{2^m}< \hat\tau\leq \frac{j}{2^m}\right\}}\frac{j}{2^m}.
\]
Then, $\hat\tau_m\downarrow\hat\tau$, $m\to\infty$. From Remark~\ref{rm:convention}, we get $\bE_x\left[e^{\int_0^{\hat\tau} g(X_s)ds +G(X_{\hat\tau})}\right]=w(x)< \infty.$ Since $g(\cdot)\geq 0$ and $0\leq G(\cdot)\leq \Vert G \Vert$, for any $T\geq 0$ and $m\in \bN$, we get
\[
e^{\int_0^{\hat\tau_m\wedge T} g(X_s)ds +G(X_{\hat\tau_m\wedge T})} \leq e^{\int_0^{\hat\tau+1} g(X_s)ds+\Vert G \Vert} \leq  e^{\Vert g\Vert+\Vert G \Vert} e^{\int_0^{\hat\tau} g(X_s)ds +G(X_{\hat\tau})}.
\]
Consequently, using bounded convergence theorem and quasi-left continuity of $(X_t)$, for any $m\in \mathbb{N}$, we get
\[
\lim_{T\to\infty} \E_x \left[ e^{\int_0^{\hat\tau_m\wedge T} g(X_s)ds +G(X_{\hat\tau_m\wedge T})}\right]= \E_x \left[\lim_{T\to\infty}e^{\int_0^{\hat\tau_m\wedge T} g(X_s)ds +G(X_{\hat\tau_m\wedge T})}\right].
\]
Consequently, we get
\begin{equation}\label{eq:conv.cl1}
 {\overline v_m}(x) \leq \E_x\left[ e^{\int_0^{\hat\tau_m} g(X_s)ds +G(X_{\hat\tau_m})}\right],
\end{equation}
where the right hand side of \eqref{eq:conv.cl1} is understood as the standard expectation, i.e. without convention~\eqref{eq:convention_stopping}. Next, using bounded convergence theorem, the fact that $\hat\tau_m\downarrow\hat\tau$, and right continuity of $(X_t)$, we get
\begin{align*}
{\overline v}(x) & = \lim_{m\to\infty} {\overline v_m}(x) \leq \lim_{m\to \infty} \bE_x \left[e^{\int_0^{\hat\tau_m} g(X_s)ds +G(X_{\hat\tau_m})}\right] \\
& = \bE_x \left[\lim_{m\to \infty} e^{\int_0^{\hat\tau_m} g(X_s)ds +G(X_{\hat\tau_m})}\right] = \bE_x \left[ e^{\int_0^{\hat\tau} g(X_s)ds +G(X_{\hat\tau})} \right]= w(x),
\end{align*}
which concludes the proof of ${\overline v}(x)\leq w(x)$, for $x\in E$.

Finally, repeating argument leading to Theorem~\ref{cor1} one can see that ${\overline v_m}$ is continuous for any $m\in \mathbb{N}$. Using Theorem~\ref{th:w_continuity}, we get that the function $w$ is continuous. As ${\overline v}\equiv w$, recalling that the convergence ${\overline v_m}(x) \to w(x)$ is monotone, and using Dini's theorem, we conclude that ${\overline v_m} \to w$ uniformly on compact sets.

\medskip
\noindent {\it Step 3.} We show that $\vert v_m-w_m\vert\to 0$ uniformly. For any $m\in\bN$ and $x\in E$, let $a_m(x):=\max(w_m(x),v_m(x))$.

First, let us assume that $x\in E$ is such that $w_m(x)\leq v_m(x)$. Let $\tau_m$ be the optimal stopping time for $w_m$. Then, using bounded convergence theorem, for $m\in\bN$, we get
\[
v_m(x)\leq \liminf_{T\to\infty}\E_x \left[ e^{\int_0^{\tau_m\wedge T} g_m(X_s)ds+G(X_{\tau_m\wedge T})}\right]=\E_x \left[e^{\int_0^{\tau_m} g_m(X_s)ds+G(X_{\tau_m})}\right].
\]
Consequently, recalling that $G_m(\cdot)\geq 0$, we get
\begin{align}
0  \leq  v_m(x)-w_m(x) & \leq\E_x \left[e^{\int_0^{\tau_m} g_m(X_s)ds} \left(e^{G(X_{\tau_m})}-e^{G_m(X_{\tau_m})}\right)\right] \nonumber\\
& \leq  \Vert e^G - e^{G_m}\Vert\E_x \left[ e^{\int_0^{\tau_m} g_m(X_s)ds+G_m(X_{\tau_m})}\right]\nonumber\\
&=  w_m(x) \Vert e^G - e^{G_m}\Vert\nonumber\\
& \leq a_m(x) \Vert e^G - e^{G_m}\Vert \label{eq:cl3}.
\end{align}
Second, let us assume that $x\in E$ is such that $w_m(x)\geq v_m(x)$. Then, as in the previous case, we get
\begin{equation}\label{eq:cl4}
0 \leq  w_m(x)-v_m(x)\leq a_m(x) \Vert e^G - e^{G_m}\Vert.
\end{equation}
Combining \eqref{eq:cl3} and \eqref{eq:cl4}, for any $x\in E$ and $m\in \bN$, we get
\[
\vert v_m(x)-w_m(x) \vert \leq a_m(x) \Vert e^G - e^{G_m}\Vert.
\]
Since $G_m\to G$ uniformly, for sufficiently large $m$, we get
\[
\Vert a_m\Vert\leq \max (\Vert w_m \Vert, \Vert v_m \Vert) \leq  \max (e^{\Vert G_m\Vert},e^{\Vert G\Vert}) \leq e^{\Vert G\Vert+1}.
\]
Combining this with inequality $|e^z-e^y|\leq e^{\max(z,y)}|z-y|$, for $z,y\in \bR$, we see that, for $m$ sufficiently large, we get
\[
\vert v_m(x)-w_m(x) \vert \leq e^{2\Vert G\Vert +2} \Vert G_m - G\Vert\to 0,\quad x\in E.
\]
Combining Step 1, Step 2, and Step 3, we conclude the proof.
\end{proof}

\bibliographystyle{siamplain}

\begin{thebibliography}{10}

\bibitem{BasSte2018}
{\sc A.~Basu and {\L}.~Stettner}, {\em Zero-sum {M}arkov games with impulse
  controls}.
\newblock submitted, 2018.

\bibitem{BauPop2018}
{\sc N.~B{\"a}uerle and A.~Popp}, {\em Risk-sensitive stopping problems for
  continuous-time {M}arkov chains}, Stochastics, 90 (2018), pp.~411--431.

\bibitem{BauRie2017b}
{\sc N.~B{\"a}uerle and U.~Rieder}, {\em Partially observable risk-sensitive
  stopping problems in discrete time}, in Modern Trends in Controlled
  Stochastic Processes II, A.~Piunovskiy, ed., Luniver Press, 2015, pp.~12--31.

\bibitem{BenLio1984}
{\sc A.~Bensoussan and J.-L. Lions}, {\em Impulse Control And Quasi-Variational
  Inequalities}, Gauthier-Villars, Montrouge, 1984.

\bibitem{Fak1970}
{\sc A.~Fakeev}, {\em Optimal stopping rules for stochastic processes with
  continuous parameter}, Theory of Probability \& Its Applications, 15 (1970),
  pp.~324--331, \url{https://doi.org/10.1137/1115039}.

\bibitem{Fak1971}
{\sc A.~Fakeev}, {\em Optimal stopping of a {M}arkov process}, Theory of
  Probability \& Its Applications, 16 (1971), pp.~694--696,
  \url{https://doi.org/10.1137/1116076}.

\bibitem{GikSko2004}
{\sc I.~Gikhman and A.~Skorokhod}, {\em The Theory of Stochastic Processes II},
  Springer, 1975.

\bibitem{HdiKar2011}
{\sc I.~Hdhiri and M.~Karouf}, {\em Risk sensitive impulse control of
  non-markovian processes}, Mathematical Methods of Operations Research, 74
  (2011), pp.~1--20.

\bibitem{JelPitSte2019b}
{\sc D.~Jelito, M.~Pitera, and {\L}.~Stettner}, {\em Long-run risk sensitive
  impulse control}.
\newblock Preprint, 2019.

\bibitem{MenRob2017}
{\sc J.~Menaldi and M.~Robin}, {\em On some impulse control problems with
  constraint}, SIAM Journal on Control and Optimization, 55 (2017),
  pp.~3204--3225, \url{https://doi.org/10.1137/16M1090302}.

\bibitem{Nag2007b}
{\sc H.~Nagai}, {\em Stopping problems of certain multiplicative functionals
  and optimal investment with transaction costs}, Applied Mathematics and
  Optimization, 55 (2007), pp.~359--384.

\bibitem{PalSte2010}
{\sc J.~Palczewski and {\L}.~Stettner}, {\em Finite horizon optimal stopping of
  time-discontinuous functionals with applications to impulse control with
  delay}, SIAM Journal on Control and Optimization, 48 (2010), pp.~4874--4909,
  \url{https://doi.org/10.1137/080737848}.

\bibitem{PalSte2017}
{\sc J.~Palczewski and {\L}.~Stettner}, {\em Impulse control maximizing average
  cost per unit time: A nonuniformly ergodic case}, SIAM Journal on Control and
  Optimization, 55 (2017), pp.~936--960.

\bibitem{PesShi2006}
{\sc G.~Peskir and A.~Shiryaev}, {\em {O}ptimal {S}topping and
  {F}ree-{B}oundary {P}roblems}, Springer, 2006.

\bibitem{PitSte2019}
{\sc M.~Pitera and {\L}.~Stettner}, {\em Long-run risk sensitive dyadic impulse
  control}, Applied Mathematics {\&} Optimization,  (2019),
  \url{https://doi.org/10.1007/s00245-019-09631-9}.

\bibitem{Rob1981}
{\sc M.~Robin}, {\em On some impulse control problems with long run average
  cost}, SIAM Journal on Control and Optimization, 19 (1981), pp.~333--358.

\bibitem{SadSte2002}
{\sc R.~Sadowy and {\L}.~Stettner}, {\em On risk-sensitive ergodic impulsive
  control of {M}arkov processes}, Applied Mathematics and Optimization, 45
  (2002), pp.~45--61.

\bibitem{Shi1978}
{\sc A.~Shiryaev}, {\em Optimal Stopping Rules}, Springer, 1978.

\bibitem{Shi2019}
{\sc A.~Shiryaev}, {\em Stochastic Disorder Problems}, Springer International
  Publishing, 2019.

\bibitem{Whi1990}
{\sc P.~Whittle}, {\em Risk-sensitive optimal control}, Wiley New York, 1990.

\bibitem{Zab1984}
{\sc J.~Zabczyk}, {\em Stopping problems in stochastic control}, in
  {P}roceedings of the {I}nternational {C}ongress of {M}athematicians, vol.~2,
  PWN, Warsaw, North-Holland, Amsterdam, 1984, pp.~1425--1437.

\end{thebibliography}

\end{document}